\documentclass[11pt]{article}

\usepackage[T1]{fontenc}
\usepackage{amsmath,amssymb,amsthm}
\usepackage{booktabs}
\usepackage[margin=0.85in]{geometry}
\usepackage{microtype}
\usepackage[hidelinks]{hyperref}
\hypersetup{
  pdftitle={Finite Configurations Cannot Generate a Constant Trace in Rule 30},
  pdfauthor={David Lee Condrey},
  pdfsubject={Rule 30 finite-support constant-trace theorems},
  pdfkeywords={cellular automata, Rule 30, temporal trace, symbolic dynamics}
}
\usepackage{tikz}
\usetikzlibrary{arrows.meta,positioning}
\usepackage{caption}
\definecolor{ruleblue}{RGB}{31,78,121}
\definecolor{rulegold}{RGB}{191,125,17}
\definecolor{ruleink}{RGB}{23,33,48}
\definecolor{ruleband}{RGB}{250,232,199}
\tikzset{
  rulecard/.style={draw=ruleblue!70!black, line width=0.7pt,
    rounded corners=4pt, fill=ruleblue!4, inner sep=7pt, align=center,
    minimum height=1.38cm},
  rulecardfill/.style={fill=ruleblue!4, rounded corners=2.4pt},
  rulecardline/.style={draw=ruleblue!70!black, line width=0.6pt,
    rounded corners=2.4pt},
  rulearrow/.style={-{Latex[length=2.8mm]}, very thick, rulegold!90!black},
}
\newcommand{\stlabel}[1]{{\small\textcolor{ruleblue!90!black}{\bfseries #1}}}

\newcommand{\stpanel}[3]{%
  \begin{tabular}[t]{@{}c@{}}
    \stlabel{#1}\\[0.5pt]
    {\scriptsize #2}\\[3.5pt]
    #3
  \end{tabular}}

\newtheorem{theorem}{Theorem}
\newtheorem{lemma}[theorem]{Lemma}
\newtheorem{corollary}[theorem]{Corollary}
\newcommand{\F}{\mathcal F}
\newcommand{\Tr}{\operatorname{Tr}}
\newcommand{\supp}{\operatorname{supp}}

\title{\vspace{-2.2em}Finite Configurations Cannot Generate a Constant Trace in Rule 30}
\author{David Lee Condrey\\WritersLogic, Inc.\\\texttt{david@writerslogic.com}}
\date{September 2026\vspace{-1.5em}}

\renewenvironment{abstract}{\small\quotation}{\endquotation}

\begin{document}
\maketitle

\begin{abstract}
For every right half we determine the unique left half whose Rule-30 central
trace is constant: an alternating tail selected by the leading one of the right
half when the initial center is $0$, and a single universal checkerboard when it
is $1$.  Every nonzero member of either fiber carries infinitely many ones, so
the zero row is the only finite configuration with a constant trace.  For
support radius $w$ the sharp maximum constant-prefix length is
$2\lceil w/2\rceil+1$ when the initial center is $0$ and $2\lfloor w/2\rfloor+2$
when it is $1$, so the maximum over both is $w+2$, attained by exactly $2^w$
configurations for even $w$ and $2^w-1$ for odd $w$.  Hence no column of a
nonzero finite Rule 30 orbit is eventually constant.
\end{abstract}



\section{Introduction}

Rule 30 is the binary cellular automaton on $\{0,1\}^{\mathbb Z}$ with local rule
\begin{equation}\label{eq:rule30}
 F(x)_i=x_{i-1}\mathbin{\oplus}(x_i\mathbin{\lor}x_{i+1}),
\end{equation}
where $\oplus$ is addition modulo $2$ and $\lor$ is disjunction.  It is left
permutive: fixing $x_i$ and $x_{i+1}$, the value $F(x)_i$ is a bijection of
$x_{i-1}$.

Wolfram's Prize Problem~1 asks whether the center trace of the lone seed is
eventually periodic \cite{wolfram}, restated as Problem~3.10 of \cite{kopra};
Kopra records the broader
finite-configuration version as open \cite{kopra-thesis}.  Jen proved that a
finite Rule 30 orbit has at most one eventually periodic column
\cite[Thm.~2b]{jen1986}\cite[Prop.~3]{jen}, and Kopra's width-two theorem does
not decide whether that exceptional column exists \cite{kopra}.  The obstruction
is structural.  Rule 90 is left permutive as well, and it admits finite
configurations whose center trace is identically zero
(Figure~\ref{fig:panels}(b)), so no argument resting on permutivity alone
separates the two rules.  The classification below turns on the OR of
\eqref{eq:rule30}, which latches, where Rule 90's XOR does not.

We settle the eventual-period-one subcase for every nonzero finite
configuration, a class strictly broader than the lone seed.  For each right half
exactly one left half gives a constant center trace: an alternating tail
selected by the leading one of the right half when the center starts at $0$, and
one universal checkerboard when it starts at $1$
(Theorems~\ref{thm:fiber} and~\ref{thm:onefiber}).  Every nonzero member of
either fiber carries infinitely many ones, so the zero row is the only finite
configuration with a constant trace (Corollaries~\ref{cor:finite}
and~\ref{cor:constant}).  A row of support radius $w$ holds its center constant
for at most $2\lceil w/2\rceil+1$ steps when the center starts at $0$ and
$2\lfloor w/2\rfloor+2$ when it starts at $1$, hence $w+2$ overall, attained by
exactly $2^w$ rows for even $w$ and $2^w-1$ for odd $w$
(Theorems~\ref{thm:horizon} and~\ref{thm:onehorizon},
Corollary~\ref{cor:bothhorizon}).  This does not resolve Prize Problem~1, and
every nonconstant eventual period remains open.

Demanding an all-zero center forces the left half cell by cell via exact
inversion, acting as a prefix-OR of the prescribed right half.  A right half
whose \emph{first} one sits at depth $3$ forces
$L_1,L_2,\ldots=0,0,1,0,1,0,1,\ldots$, ones at every subsequent odd depth.  This
one-sided bijectivity mirrors Rowland's unique-history construction
\cite[Prop.~1]{rowland}.  The latching OR in \eqref{eq:rule30} makes the
classification Rule-30-specific, unlike Rule 90
(Figure~\ref{fig:panels}, Section~\ref{sec:finite}).

\begin{figure}[tbp]
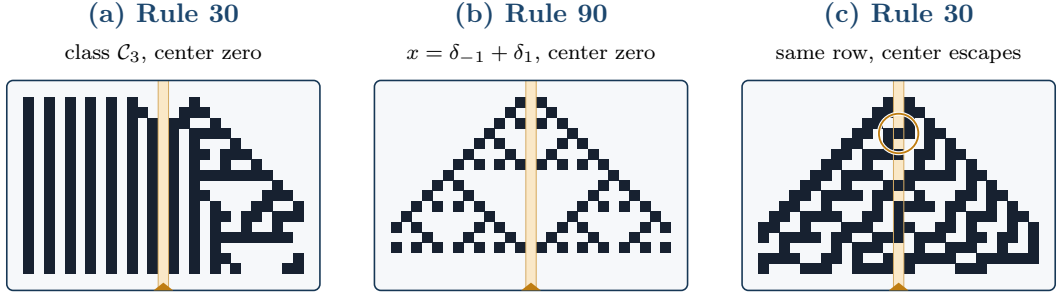

\centering
\begin{tabular}{@{}c@{\hspace{7mm}}c@{\hspace{7mm}}c@{}}
\stpanel{(a) Rule 30}{class $\mathcal C_3$, center zero}{%
  \begin{tikzpicture}[scale=2.365]\input{figures/fig-fiber}\end{tikzpicture}}
&
\stpanel{(b) Rule 90}{$x=\delta_{-1}+\delta_{1}$, center zero}{%
  \begin{tikzpicture}[scale=2.365]  \fill[rulecardfill] (-0.8468,0.0928) rectangle (0.9048,-1.0788);
  \fill[ruleband] (0.0000,0.0928) rectangle (0.0580,-1.0788);
  \fill[ruleink] (0.0580,0.0000) rectangle ++(0.0580,-0.0580);
  \fill[ruleink] (-0.0580,0.0000) rectangle ++(0.0580,-0.0580);
  \fill[ruleink] (0.1160,-0.0580) rectangle ++(0.0580,-0.0580);
  \fill[ruleink] (-0.1160,-0.0580) rectangle ++(0.0580,-0.0580);
  \fill[ruleink] (0.0580,-0.1160) rectangle ++(0.0580,-0.0580);
  \fill[ruleink] (0.1740,-0.1160) rectangle ++(0.0580,-0.0580);
  \fill[ruleink] (-0.1740,-0.1160) rectangle ++(0.0580,-0.0580);
  \fill[ruleink] (-0.0580,-0.1160) rectangle ++(0.0580,-0.0580);
  \fill[ruleink] (-0.2320,-0.1740) rectangle ++(0.0580,-0.0580);
  \fill[ruleink] (0.2320,-0.1740) rectangle ++(0.0580,-0.0580);
  \fill[ruleink] (0.1740,-0.2320) rectangle ++(0.0580,-0.0580);
  \fill[ruleink] (0.2900,-0.2320) rectangle ++(0.0580,-0.0580);
  \fill[ruleink] (-0.2900,-0.2320) rectangle ++(0.0580,-0.0580);
  \fill[ruleink] (-0.1740,-0.2320) rectangle ++(0.0580,-0.0580);
  \fill[ruleink] (-0.3480,-0.2900) rectangle ++(0.0580,-0.0580);
  \fill[ruleink] (0.1160,-0.2900) rectangle ++(0.0580,-0.0580);
  \fill[ruleink] (-0.1160,-0.2900) rectangle ++(0.0580,-0.0580);
  \fill[ruleink] (0.3480,-0.2900) rectangle ++(0.0580,-0.0580);
  \fill[ruleink] (0.0580,-0.3480) rectangle ++(0.0580,-0.0580);
  \fill[ruleink] (0.1740,-0.3480) rectangle ++(0.0580,-0.0580);
  \fill[ruleink] (0.2900,-0.3480) rectangle ++(0.0580,-0.0580);
  \fill[ruleink] (0.4060,-0.3480) rectangle ++(0.0580,-0.0580);
  \fill[ruleink] (-0.4060,-0.3480) rectangle ++(0.0580,-0.0580);
  \fill[ruleink] (-0.2900,-0.3480) rectangle ++(0.0580,-0.0580);
  \fill[ruleink] (-0.1740,-0.3480) rectangle ++(0.0580,-0.0580);
  \fill[ruleink] (-0.0580,-0.3480) rectangle ++(0.0580,-0.0580);
  \fill[ruleink] (-0.4640,-0.4060) rectangle ++(0.0580,-0.0580);
  \fill[ruleink] (0.4640,-0.4060) rectangle ++(0.0580,-0.0580);
  \fill[ruleink] (-0.4060,-0.4640) rectangle ++(0.0580,-0.0580);
  \fill[ruleink] (0.4060,-0.4640) rectangle ++(0.0580,-0.0580);
  \fill[ruleink] (0.5220,-0.4640) rectangle ++(0.0580,-0.0580);
  \fill[ruleink] (-0.5220,-0.4640) rectangle ++(0.0580,-0.0580);
  \fill[ruleink] (-0.3480,-0.5220) rectangle ++(0.0580,-0.0580);
  \fill[ruleink] (0.5800,-0.5220) rectangle ++(0.0580,-0.0580);
  \fill[ruleink] (-0.5800,-0.5220) rectangle ++(0.0580,-0.0580);
  \fill[ruleink] (0.3480,-0.5220) rectangle ++(0.0580,-0.0580);
  \fill[ruleink] (0.2900,-0.5800) rectangle ++(0.0580,-0.0580);
  \fill[ruleink] (0.4060,-0.5800) rectangle ++(0.0580,-0.0580);
  \fill[ruleink] (0.5220,-0.5800) rectangle ++(0.0580,-0.0580);
  \fill[ruleink] (0.6380,-0.5800) rectangle ++(0.0580,-0.0580);
  \fill[ruleink] (-0.6380,-0.5800) rectangle ++(0.0580,-0.0580);
  \fill[ruleink] (-0.5220,-0.5800) rectangle ++(0.0580,-0.0580);
  \fill[ruleink] (-0.4060,-0.5800) rectangle ++(0.0580,-0.0580);
  \fill[ruleink] (-0.2900,-0.5800) rectangle ++(0.0580,-0.0580);
  \fill[ruleink] (0.2320,-0.6380) rectangle ++(0.0580,-0.0580);
  \fill[ruleink] (0.6960,-0.6380) rectangle ++(0.0580,-0.0580);
  \fill[ruleink] (-0.6960,-0.6380) rectangle ++(0.0580,-0.0580);
  \fill[ruleink] (-0.2320,-0.6380) rectangle ++(0.0580,-0.0580);
  \fill[ruleink] (0.1740,-0.6960) rectangle ++(0.0580,-0.0580);
  \fill[ruleink] (0.2900,-0.6960) rectangle ++(0.0580,-0.0580);
  \fill[ruleink] (0.6380,-0.6960) rectangle ++(0.0580,-0.0580);
  \fill[ruleink] (0.7540,-0.6960) rectangle ++(0.0580,-0.0580);
  \fill[ruleink] (-0.7540,-0.6960) rectangle ++(0.0580,-0.0580);
  \fill[ruleink] (-0.6380,-0.6960) rectangle ++(0.0580,-0.0580);
  \fill[ruleink] (-0.2900,-0.6960) rectangle ++(0.0580,-0.0580);
  \fill[ruleink] (-0.1740,-0.6960) rectangle ++(0.0580,-0.0580);
  \fill[ruleink] (0.1160,-0.7540) rectangle ++(0.0580,-0.0580);
  \fill[ruleink] (0.3480,-0.7540) rectangle ++(0.0580,-0.0580);
  \fill[ruleink] (0.5800,-0.7540) rectangle ++(0.0580,-0.0580);
  \fill[ruleink] (-0.5800,-0.7540) rectangle ++(0.0580,-0.0580);
  \fill[ruleink] (-0.3480,-0.7540) rectangle ++(0.0580,-0.0580);
  \fill[ruleink] (-0.1160,-0.7540) rectangle ++(0.0580,-0.0580);
  \fill[ruleink] (0.0580,-0.8120) rectangle ++(0.0580,-0.0580);
  \fill[ruleink] (0.1740,-0.8120) rectangle ++(0.0580,-0.0580);
  \fill[ruleink] (0.2900,-0.8120) rectangle ++(0.0580,-0.0580);
  \fill[ruleink] (0.4060,-0.8120) rectangle ++(0.0580,-0.0580);
  \fill[ruleink] (0.5220,-0.8120) rectangle ++(0.0580,-0.0580);
  \fill[ruleink] (0.6380,-0.8120) rectangle ++(0.0580,-0.0580);
  \fill[ruleink] (0.7540,-0.8120) rectangle ++(0.0580,-0.0580);
  \fill[ruleink] (-0.7540,-0.8120) rectangle ++(0.0580,-0.0580);
  \fill[ruleink] (-0.6380,-0.8120) rectangle ++(0.0580,-0.0580);
  \fill[ruleink] (-0.5220,-0.8120) rectangle ++(0.0580,-0.0580);
  \fill[ruleink] (-0.4060,-0.8120) rectangle ++(0.0580,-0.0580);
  \fill[ruleink] (-0.2900,-0.8120) rectangle ++(0.0580,-0.0580);
  \fill[ruleink] (-0.1740,-0.8120) rectangle ++(0.0580,-0.0580);
  \fill[ruleink] (-0.0580,-0.8120) rectangle ++(0.0580,-0.0580);
  \draw[rulegold!65, line width=0.3pt] (0.0000,0.0928) -- (0.0000,-1.0788);
  \draw[rulegold!65, line width=0.3pt] (0.0580,0.0928) -- (0.0580,-1.0788);
  \draw[rulecardline] (-0.8468,0.0928) rectangle (0.9048,-1.0788);
  \fill[rulegold] (-0.0261,-1.0788) -- (0.0841,-1.0788) -- (0.0290,-1.0353) -- cycle;\end{tikzpicture}}
&
\stpanel{(c) Rule 30}{same row, center escapes}{%
  \begin{tikzpicture}[scale=2.365]\input{figures/fig-rule30}\end{tikzpicture}}
\end{tabular}
\\[3mm]
\caption{A zero center trace forces an infinite left tail under Rule 30, but not
under Rule 90.  Time runs downward; the gold band and caret mark column $0$.
(a) The left half forced by a right half of class $\mathcal C_3$ alternates,
hence carries infinitely many ones, which is what Corollary~\ref{cor:finite}
exploits.  (b) and (c) run the single finite row $x=\delta_{-1}+\delta_1$ under
Rule 90 and under Rule 30: its center stays zero under Rule 90 and escapes under
Rule 30.  The only difference between the two rules is the OR latch of
\eqref{eq:rule30}.}
\label{fig:panels}
\end{figure}

Jen gives necessary and sufficient conditions for a finite initial condition to
generate an eventually-one temporal sequence \cite[Thm.~7a]{jen1986}.  Rule 30
meets the hypothesis $a_0=0$, $a_1=a_4=1$ but none of the four alternatives,
which require respectively $a_6=a_7=1$, $a_6=1$, $a_7=1$, or $a_5=1$, while
Rule 30 has $a_5=a_6=a_7=0$; necessity therefore already excludes an
eventually-one column for every finite Rule 30 configuration, and we claim no
novelty for that exclusion.  New here are the two exact fixed-locus fibers, the
explicit zero exclusion, and the sharp horizons with extremizer counts.  Jen
remarks that symmetric results hold for constant-zero sequences but states
neither a zero-fiber classification nor any of these quantitative conclusions.
Finally, deciding whether a subconfiguration of length $K$ has a $K$-step
preimage is NP-complete already for a nearest-neighbor rule \cite{green}, and
one-step inversion has been solved for the \emph{additive} class over
$\mathbb Z_p$ by an operator analogous to backward integration \cite{voorhees};
Theorem~\ref{thm:fiber} inverts a constraint imposed over all future time, for a
rule that is not additive.

\section{The constant-trace fibers}

The trace and the two halves of a row are written as follows.

\begin{equation}\label{eq:notation}
 c_t:=\Tr(x)_t=F^t(x)_0,\qquad R_j:=x_j,\qquad L_j:=x_{-j}\quad(j\geq0)
\end{equation}

Throughout, $\F$ denotes the finite-support configurations.

The trace sees only column $0$, but under a left-permutive rule a discrepancy
left of the origin propagates rightward along a diagonal and must eventually
reach it.  The classification rests on this one fact, which uses only locality
and left permutivity, not the specific form of \eqref{eq:rule30}.

\begin{lemma}[Triangular uniqueness]\label{lem:unique}
If two configurations agree at every coordinate $i\geq0$ and have the same
central trace, then they are equal.
\end{lemma}
\begin{proof}
If $-n$ is their rightmost differing coordinate, locality preserves agreement
to its right through time $n$.  Left permutivity preserves the discrepancy
along $(-n,0),(-n+1,1),\ldots,(0,n)$, contradicting equal traces.  In its
finite form, if the rows agree on $i\geq0$ and their traces agree for
$0\leq t\leq r$, the same induction forces agreement at
$-r,-r+1,\ldots,-1$.
\end{proof}

Uniqueness alone does not say what the forced left half looks like.  The next
result computes it: the rows with zero trace form one invariant class
$\mathcal C_m$ for each position $m$ of the first right-hand one, each class is
carried to its predecessor by $F$, and the left half is a prefix-OR of the
right.  Figure~\ref{fig:panels}(a) is a member of $\mathcal C_3$.

\begin{theorem}[Zero-trace fiber]\label{thm:fiber}
For a prescribed right half with $R_0=0$, the zero-trace-compatible left half
is unique.  Suppose the positive half is nonzero, and let
$m=\min\{j\geq1:R_j=1\}$.  The forced left half is then the following, with
$R_j$ for $j>m$ arbitrary.

\begin{equation}\label{eq:pattern}
 L_j=0\ (j<m),\qquad L_m=1,\qquad L_j=j\bmod2\ (j>m)
\end{equation}

If instead the positive half is zero, the unique row is zero.  The same
classification in prefix-OR form is the following.

\begin{align}
 L_{2k+1}&=\bigvee_{j=1}^{2k+1}R_j &&(k\geq0) \label{eq:odd}\\[-2pt]
 L_{2k}&=R_{2k}\land\neg\!\bigvee_{j=1}^{2k-1}R_j &&(k\geq1) \label{eq:even}
\end{align}
\end{theorem}
\begin{proof}
For $m\geq1$, let $\mathcal C_m$ consist of rows satisfying these conditions.
One Rule 30 step gives the three identities below.

\begin{equation}\label{eq:updates}
\begin{aligned}
 x'_0&=L_1\oplus R_1\\
 R'_j&=R_{j-1}\oplus(R_j\lor R_{j+1})\\
 L'_j&=L_{j+1}\oplus(L_j\lor L_{j-1})
\end{aligned}
\end{equation}

Suppose $m>1$.  These identities then give the following.

\begin{equation}\label{eq:step}
 x'_0=0,\quad R'_j=L'_j=0\ (j<m-1),\quad R'_{m-1}=L'_{m-1}=1,
 \quad L'_m=m\bmod2
\end{equation}

For $j>m$, including $j=m+1$ because $L_m=1$, the OR in the last update is
one, so $L'_j=(1-(j\bmod2))\oplus1=j\bmod2$.  Thus
$F(\mathcal C_m)\subseteq\mathcal C_{m-1}$.  For $m=1$, the alternating left
half is fixed, $x'_0=1\oplus1=0$, and
$R'_1=0\oplus(1\lor R_2)=1$ independently of the right tail; hence
$F(\mathcal C_1)\subseteq\mathcal C_1$.  Every $\mathcal C_m$ therefore
realizes the zero trace.

A nonzero right half selects a realization by its first one, and
Lemma~\ref{lem:unique} makes it unique; a zero right half forces a zero left
half by the same lemma against the zero row.
\end{proof}

The other constant trace has a simpler fiber: one row per right half, and the
left half no longer depends on which right half it is.

\begin{theorem}[All-one fiber]\label{thm:onefiber}
For a prescribed right half with $R_0=1$, the left half compatible with an
identically one trace is unique and is the same for every such right half.

\begin{equation}\label{eq:checker}
 L_j=1\iff j\text{ is positive and even}
\end{equation}
\end{theorem}
\begin{proof}
Write $L_0:=x_0=R_0=1$, so that the third identity of \eqref{eq:updates},
$L'_j=L_{j+1}\oplus(L_j\lor L_{j-1})$, applies at every $j\geq1$.  Under
\eqref{eq:checker} it gives $0\oplus(1\lor0)=1$ for even $j$ and
$1\oplus(0\lor1)=0$ for odd $j$, so the left half is fixed by $F$ whatever the
right half is.  At the origin, $F(x)_0=L_1\oplus(L_0\lor R_1)=0\oplus1=1$
independently of $R_1$, so the trace is identically one and $L_0=1$ is restored
at every step.  Lemma~\ref{lem:unique} makes the left half unique.
\end{proof}

\section{Finite configurations}\label{sec:finite}

Every nonzero row in Theorem~\ref{thm:fiber} carries ones arbitrarily far to
the left, so none of them is finitely supported.  That is the whole of the next
statement, once one checks that a finite orbit cannot become the zero row.

\begin{corollary}\label{cor:finite}
$\Tr^{-1}(0^{\mathbb N})\cap\F=\{0\}$.  Moreover, no column of the orbit of a
nonzero $x\in\F$ is eventually zero.
\end{corollary}
\begin{proof}
Every nonzero row in Theorem~\ref{thm:fiber} has ones at all sufficiently
large odd left depths.  If a column of a finite orbit were zero from time
$T$ onward, translation would make $F^T(x)$ a finite zero-trace row.  It is
nonzero: if $[a,b]$ is the support interval of a nonzero finite row, its next
row has ones at $a-1$ and $b+1$ because Rule 30 maps $001$ and $100$ to one.
\end{proof}

Corollary~\ref{cor:finite} is not a consequence of left permutivity alone:
Rule 90 is also left permutive, yet the finite row with ones at $\pm1$ keeps a
zero center forever (Figure~\ref{fig:panels}(b,c)).  The Rule-30-specific step
is the OR latch $R'_1=1\lor R_2=1$ in the proof of Theorem~\ref{thm:fiber},
which has no analogue in an additive rule.

The same latch upgrades the corollary from one constant to both, with no
further machinery.  Theorem~\ref{thm:onefiber} proves the second half again and
independently, since \eqref{eq:checker} has ones at every positive even depth;
the argument below is shorter and is the one to quote.

\begin{corollary}\label{cor:constant}
No column of the orbit of a nonzero $x\in\F$ is eventually constant.
\end{corollary}
\begin{proof}
The eventually-zero case is Corollary~\ref{cor:finite}.  Suppose instead some
column is eventually one; translating, take it to be column $0$, so $c_t=1$ for
all $t\geq T$.  Applying \eqref{eq:rule30} at the origin gives
$l_t=c_{t+1}\oplus(c_t\lor r_t)$, and $c_t=1$ collapses the disjunction, so
$l_t=1\oplus c_{t+1}$.  For $t\geq T$ we also have $c_{t+1}=1$, hence $l_t=0$.
Column $-1$ is then zero from time $T$ onward, which
Corollary~\ref{cor:finite} forbids.
\end{proof}

\begin{figure}[t]
\centering
\begin{tabular}{@{}c@{\hspace{5mm}}c@{\hspace{5mm}}c@{}}
\stlabel{$w=1$} & \stlabel{$w=2$} & \stlabel{$w=3$}
\\[3pt]
\begin{tikzpicture}[scale=2.8224]  \fill[rulecardfill] (-0.3828,0.0928) rectangle (0.4408,-0.3248);
  \fill[ruleband] (0.0000,0.0928) rectangle (0.0580,-0.3248);
  \fill[ruleink] (0.0580,0.0000) rectangle ++(0.0580,-0.0580);
  \fill[ruleink] (-0.0580,0.0000) rectangle ++(0.0580,-0.0580);
  \fill[ruleink] (0.0580,-0.0580) rectangle ++(0.0580,-0.0580);
  \fill[ruleink] (0.1160,-0.0580) rectangle ++(0.0580,-0.0580);
  \fill[ruleink] (-0.1160,-0.0580) rectangle ++(0.0580,-0.0580);
  \fill[ruleink] (-0.0580,-0.0580) rectangle ++(0.0580,-0.0580);
  \fill[ruleink] (0.0580,-0.1160) rectangle ++(0.0580,-0.0580);
  \fill[ruleink] (0.1740,-0.1160) rectangle ++(0.0580,-0.0580);
  \fill[ruleink] (-0.1740,-0.1160) rectangle ++(0.0580,-0.0580);
  \fill[ruleink] (-0.1160,-0.1160) rectangle ++(0.0580,-0.0580);
  \fill[ruleink] (0.0000,-0.1740) rectangle ++(0.0580,-0.0580);
  \fill[ruleink] (0.0580,-0.1740) rectangle ++(0.0580,-0.0580);
  \fill[ruleink] (0.1740,-0.1740) rectangle ++(0.0580,-0.0580);
  \fill[ruleink] (0.2320,-0.1740) rectangle ++(0.0580,-0.0580);
  \fill[ruleink] (-0.2320,-0.1740) rectangle ++(0.0580,-0.0580);
  \fill[ruleink] (-0.1740,-0.1740) rectangle ++(0.0580,-0.0580);
  \fill[ruleink] (-0.0580,-0.1740) rectangle ++(0.0580,-0.0580);
  \draw[rulegold!65, line width=0.3pt] (0.0000,0.0928) -- (0.0000,-0.3248);
  \draw[rulegold!65, line width=0.3pt] (0.0580,0.0928) -- (0.0580,-0.3248);
  \draw[rulecardline] (-0.3828,0.0928) rectangle (0.4408,-0.3248);
  \fill[rulegold] (-0.0261,-0.3248) -- (0.0841,-0.3248) -- (0.0290,-0.2813) -- cycle;
  \draw[white, line width=1.5pt] (0.0290,-0.2030) circle (0.1102);
  \draw[rulegold, line width=0.7pt] (0.0290,-0.2030) circle (0.1102);\end{tikzpicture} &
\begin{tikzpicture}[scale=2.8224]  \fill[rulecardfill] (-0.4408,0.0928) rectangle (0.4988,-0.3248);
  \fill[ruleband] (0.0000,0.0928) rectangle (0.0580,-0.3248);
  \fill[ruleink] (0.1160,0.0000) rectangle ++(0.0580,-0.0580);
  \fill[ruleink] (-0.1160,0.0000) rectangle ++(0.0580,-0.0580);
  \fill[ruleink] (0.0580,-0.0580) rectangle ++(0.0580,-0.0580);
  \fill[ruleink] (0.1160,-0.0580) rectangle ++(0.0580,-0.0580);
  \fill[ruleink] (0.1740,-0.0580) rectangle ++(0.0580,-0.0580);
  \fill[ruleink] (-0.1160,-0.0580) rectangle ++(0.0580,-0.0580);
  \fill[ruleink] (-0.1740,-0.0580) rectangle ++(0.0580,-0.0580);
  \fill[ruleink] (-0.0580,-0.0580) rectangle ++(0.0580,-0.0580);
  \fill[ruleink] (0.0580,-0.1160) rectangle ++(0.0580,-0.0580);
  \fill[ruleink] (0.2320,-0.1160) rectangle ++(0.0580,-0.0580);
  \fill[ruleink] (-0.2320,-0.1160) rectangle ++(0.0580,-0.0580);
  \fill[ruleink] (-0.1740,-0.1160) rectangle ++(0.0580,-0.0580);
  \fill[ruleink] (0.0000,-0.1740) rectangle ++(0.0580,-0.0580);
  \fill[ruleink] (0.0580,-0.1740) rectangle ++(0.0580,-0.0580);
  \fill[ruleink] (0.1160,-0.1740) rectangle ++(0.0580,-0.0580);
  \fill[ruleink] (0.1740,-0.1740) rectangle ++(0.0580,-0.0580);
  \fill[ruleink] (0.2320,-0.1740) rectangle ++(0.0580,-0.0580);
  \fill[ruleink] (0.2900,-0.1740) rectangle ++(0.0580,-0.0580);
  \fill[ruleink] (-0.2900,-0.1740) rectangle ++(0.0580,-0.0580);
  \fill[ruleink] (-0.2320,-0.1740) rectangle ++(0.0580,-0.0580);
  \fill[ruleink] (-0.1160,-0.1740) rectangle ++(0.0580,-0.0580);
  \draw[rulegold!65, line width=0.3pt] (0.0000,0.0928) -- (0.0000,-0.3248);
  \draw[rulegold!65, line width=0.3pt] (0.0580,0.0928) -- (0.0580,-0.3248);
  \draw[rulecardline] (-0.4408,0.0928) rectangle (0.4988,-0.3248);
  \fill[rulegold] (-0.0261,-0.3248) -- (0.0841,-0.3248) -- (0.0290,-0.2813) -- cycle;
  \draw[white, line width=1.5pt] (0.0290,-0.2030) circle (0.1102);
  \draw[rulegold, line width=0.7pt] (0.0290,-0.2030) circle (0.1102);\end{tikzpicture} &
\begin{tikzpicture}[scale=2.8224]  \fill[rulecardfill] (-0.6148,0.0928) rectangle (0.6728,-0.4408);
  \fill[ruleband] (0.0000,0.0928) rectangle (0.0580,-0.4408);
  \fill[ruleink] (0.1740,0.0000) rectangle ++(0.0580,-0.0580);
  \fill[ruleink] (-0.1740,0.0000) rectangle ++(0.0580,-0.0580);
  \fill[ruleink] (0.1160,-0.0580) rectangle ++(0.0580,-0.0580);
  \fill[ruleink] (0.1740,-0.0580) rectangle ++(0.0580,-0.0580);
  \fill[ruleink] (0.2320,-0.0580) rectangle ++(0.0580,-0.0580);
  \fill[ruleink] (-0.2320,-0.0580) rectangle ++(0.0580,-0.0580);
  \fill[ruleink] (-0.1740,-0.0580) rectangle ++(0.0580,-0.0580);
  \fill[ruleink] (-0.1160,-0.0580) rectangle ++(0.0580,-0.0580);
  \fill[ruleink] (0.0580,-0.1160) rectangle ++(0.0580,-0.0580);
  \fill[ruleink] (0.1160,-0.1160) rectangle ++(0.0580,-0.0580);
  \fill[ruleink] (0.2900,-0.1160) rectangle ++(0.0580,-0.0580);
  \fill[ruleink] (-0.2900,-0.1160) rectangle ++(0.0580,-0.0580);
  \fill[ruleink] (-0.2320,-0.1160) rectangle ++(0.0580,-0.0580);
  \fill[ruleink] (-0.0580,-0.1160) rectangle ++(0.0580,-0.0580);
  \fill[ruleink] (0.0580,-0.1740) rectangle ++(0.0580,-0.0580);
  \fill[ruleink] (0.1740,-0.1740) rectangle ++(0.0580,-0.0580);
  \fill[ruleink] (0.2320,-0.1740) rectangle ++(0.0580,-0.0580);
  \fill[ruleink] (0.2900,-0.1740) rectangle ++(0.0580,-0.0580);
  \fill[ruleink] (0.3480,-0.1740) rectangle ++(0.0580,-0.0580);
  \fill[ruleink] (-0.0580,-0.1740) rectangle ++(0.0580,-0.0580);
  \fill[ruleink] (-0.3480,-0.1740) rectangle ++(0.0580,-0.0580);
  \fill[ruleink] (-0.2900,-0.1740) rectangle ++(0.0580,-0.0580);
  \fill[ruleink] (-0.1740,-0.1740) rectangle ++(0.0580,-0.0580);
  \fill[ruleink] (-0.1160,-0.1740) rectangle ++(0.0580,-0.0580);
  \fill[ruleink] (0.0580,-0.2320) rectangle ++(0.0580,-0.0580);
  \fill[ruleink] (0.1740,-0.2320) rectangle ++(0.0580,-0.0580);
  \fill[ruleink] (0.4060,-0.2320) rectangle ++(0.0580,-0.0580);
  \fill[ruleink] (-0.4060,-0.2320) rectangle ++(0.0580,-0.0580);
  \fill[ruleink] (-0.3480,-0.2320) rectangle ++(0.0580,-0.0580);
  \fill[ruleink] (-0.1740,-0.2320) rectangle ++(0.0580,-0.0580);
  \fill[ruleink] (0.0000,-0.2900) rectangle ++(0.0580,-0.0580);
  \fill[ruleink] (0.0580,-0.2900) rectangle ++(0.0580,-0.0580);
  \fill[ruleink] (0.1740,-0.2900) rectangle ++(0.0580,-0.0580);
  \fill[ruleink] (0.2320,-0.2900) rectangle ++(0.0580,-0.0580);
  \fill[ruleink] (0.3480,-0.2900) rectangle ++(0.0580,-0.0580);
  \fill[ruleink] (0.4060,-0.2900) rectangle ++(0.0580,-0.0580);
  \fill[ruleink] (0.4640,-0.2900) rectangle ++(0.0580,-0.0580);
  \fill[ruleink] (-0.4640,-0.2900) rectangle ++(0.0580,-0.0580);
  \fill[ruleink] (-0.4060,-0.2900) rectangle ++(0.0580,-0.0580);
  \fill[ruleink] (-0.2900,-0.2900) rectangle ++(0.0580,-0.0580);
  \fill[ruleink] (-0.2320,-0.2900) rectangle ++(0.0580,-0.0580);
  \fill[ruleink] (-0.1740,-0.2900) rectangle ++(0.0580,-0.0580);
  \fill[ruleink] (-0.1160,-0.2900) rectangle ++(0.0580,-0.0580);
  \draw[rulegold!65, line width=0.3pt] (0.0000,0.0928) -- (0.0000,-0.4408);
  \draw[rulegold!65, line width=0.3pt] (0.0580,0.0928) -- (0.0580,-0.4408);
  \draw[rulecardline] (-0.6148,0.0928) rectangle (0.6728,-0.4408);
  \fill[rulegold] (-0.0261,-0.4408) -- (0.0841,-0.4408) -- (0.0290,-0.3973) -- cycle;
  \draw[white, line width=1.5pt] (0.0290,-0.3190) circle (0.1102);
  \draw[rulegold, line width=0.7pt] (0.0290,-0.3190) circle (0.1102);\end{tikzpicture}
\\[3.5mm]
\stlabel{$w=4$} & \stlabel{$w=5$} & \stlabel{$w=6$}
\\[3pt]
\begin{tikzpicture}[scale=2.8224]  \fill[rulecardfill] (-0.6728,0.0928) rectangle (0.7308,-0.4408);
  \fill[ruleband] (0.0000,0.0928) rectangle (0.0580,-0.4408);
  \fill[ruleink] (-0.2320,0.0000) rectangle ++(0.0580,-0.0580);
  \fill[ruleink] (0.2320,0.0000) rectangle ++(0.0580,-0.0580);
  \fill[ruleink] (0.1740,-0.0580) rectangle ++(0.0580,-0.0580);
  \fill[ruleink] (0.2320,-0.0580) rectangle ++(0.0580,-0.0580);
  \fill[ruleink] (0.2900,-0.0580) rectangle ++(0.0580,-0.0580);
  \fill[ruleink] (-0.2900,-0.0580) rectangle ++(0.0580,-0.0580);
  \fill[ruleink] (-0.2320,-0.0580) rectangle ++(0.0580,-0.0580);
  \fill[ruleink] (-0.1740,-0.0580) rectangle ++(0.0580,-0.0580);
  \fill[ruleink] (0.1160,-0.1160) rectangle ++(0.0580,-0.0580);
  \fill[ruleink] (0.1740,-0.1160) rectangle ++(0.0580,-0.0580);
  \fill[ruleink] (0.3480,-0.1160) rectangle ++(0.0580,-0.0580);
  \fill[ruleink] (-0.3480,-0.1160) rectangle ++(0.0580,-0.0580);
  \fill[ruleink] (-0.2900,-0.1160) rectangle ++(0.0580,-0.0580);
  \fill[ruleink] (-0.1160,-0.1160) rectangle ++(0.0580,-0.0580);
  \fill[ruleink] (0.0580,-0.1740) rectangle ++(0.0580,-0.0580);
  \fill[ruleink] (0.1160,-0.1740) rectangle ++(0.0580,-0.0580);
  \fill[ruleink] (0.2320,-0.1740) rectangle ++(0.0580,-0.0580);
  \fill[ruleink] (0.2900,-0.1740) rectangle ++(0.0580,-0.0580);
  \fill[ruleink] (0.3480,-0.1740) rectangle ++(0.0580,-0.0580);
  \fill[ruleink] (0.4060,-0.1740) rectangle ++(0.0580,-0.0580);
  \fill[ruleink] (-0.0580,-0.1740) rectangle ++(0.0580,-0.0580);
  \fill[ruleink] (-0.4060,-0.1740) rectangle ++(0.0580,-0.0580);
  \fill[ruleink] (-0.3480,-0.1740) rectangle ++(0.0580,-0.0580);
  \fill[ruleink] (-0.2320,-0.1740) rectangle ++(0.0580,-0.0580);
  \fill[ruleink] (-0.1740,-0.1740) rectangle ++(0.0580,-0.0580);
  \fill[ruleink] (-0.1160,-0.1740) rectangle ++(0.0580,-0.0580);
  \fill[ruleink] (0.0580,-0.2320) rectangle ++(0.0580,-0.0580);
  \fill[ruleink] (0.2320,-0.2320) rectangle ++(0.0580,-0.0580);
  \fill[ruleink] (0.4640,-0.2320) rectangle ++(0.0580,-0.0580);
  \fill[ruleink] (-0.4640,-0.2320) rectangle ++(0.0580,-0.0580);
  \fill[ruleink] (-0.2320,-0.2320) rectangle ++(0.0580,-0.0580);
  \fill[ruleink] (-0.4060,-0.2320) rectangle ++(0.0580,-0.0580);
  \fill[ruleink] (0.0000,-0.2900) rectangle ++(0.0580,-0.0580);
  \fill[ruleink] (0.0580,-0.2900) rectangle ++(0.0580,-0.0580);
  \fill[ruleink] (0.1160,-0.2900) rectangle ++(0.0580,-0.0580);
  \fill[ruleink] (0.1740,-0.2900) rectangle ++(0.0580,-0.0580);
  \fill[ruleink] (0.2320,-0.2900) rectangle ++(0.0580,-0.0580);
  \fill[ruleink] (0.2900,-0.2900) rectangle ++(0.0580,-0.0580);
  \fill[ruleink] (0.4060,-0.2900) rectangle ++(0.0580,-0.0580);
  \fill[ruleink] (0.4640,-0.2900) rectangle ++(0.0580,-0.0580);
  \fill[ruleink] (0.5220,-0.2900) rectangle ++(0.0580,-0.0580);
  \fill[ruleink] (-0.5220,-0.2900) rectangle ++(0.0580,-0.0580);
  \fill[ruleink] (-0.4640,-0.2900) rectangle ++(0.0580,-0.0580);
  \fill[ruleink] (-0.3480,-0.2900) rectangle ++(0.0580,-0.0580);
  \fill[ruleink] (-0.2900,-0.2900) rectangle ++(0.0580,-0.0580);
  \fill[ruleink] (-0.2320,-0.2900) rectangle ++(0.0580,-0.0580);
  \fill[ruleink] (-0.1740,-0.2900) rectangle ++(0.0580,-0.0580);
  \draw[rulegold!65, line width=0.3pt] (0.0000,0.0928) -- (0.0000,-0.4408);
  \draw[rulegold!65, line width=0.3pt] (0.0580,0.0928) -- (0.0580,-0.4408);
  \draw[rulecardline] (-0.6728,0.0928) rectangle (0.7308,-0.4408);
  \fill[rulegold] (-0.0261,-0.4408) -- (0.0841,-0.4408) -- (0.0290,-0.3973) -- cycle;
  \draw[white, line width=1.5pt] (0.0290,-0.3190) circle (0.1102);
  \draw[rulegold, line width=0.7pt] (0.0290,-0.3190) circle (0.1102);\end{tikzpicture} &
\begin{tikzpicture}[scale=2.8224]  \fill[rulecardfill] (-0.8468,0.0928) rectangle (0.9048,-0.5568);
  \fill[ruleband] (0.0000,0.0928) rectangle (0.0580,-0.5568);
  \fill[ruleink] (-0.2900,0.0000) rectangle ++(0.0580,-0.0580);
  \fill[ruleink] (0.2900,0.0000) rectangle ++(0.0580,-0.0580);
  \fill[ruleink] (0.2320,-0.0580) rectangle ++(0.0580,-0.0580);
  \fill[ruleink] (0.2900,-0.0580) rectangle ++(0.0580,-0.0580);
  \fill[ruleink] (0.3480,-0.0580) rectangle ++(0.0580,-0.0580);
  \fill[ruleink] (-0.3480,-0.0580) rectangle ++(0.0580,-0.0580);
  \fill[ruleink] (-0.2900,-0.0580) rectangle ++(0.0580,-0.0580);
  \fill[ruleink] (-0.2320,-0.0580) rectangle ++(0.0580,-0.0580);
  \fill[ruleink] (0.1740,-0.1160) rectangle ++(0.0580,-0.0580);
  \fill[ruleink] (0.2320,-0.1160) rectangle ++(0.0580,-0.0580);
  \fill[ruleink] (0.4060,-0.1160) rectangle ++(0.0580,-0.0580);
  \fill[ruleink] (-0.4060,-0.1160) rectangle ++(0.0580,-0.0580);
  \fill[ruleink] (-0.3480,-0.1160) rectangle ++(0.0580,-0.0580);
  \fill[ruleink] (-0.1740,-0.1160) rectangle ++(0.0580,-0.0580);
  \fill[ruleink] (0.1160,-0.1740) rectangle ++(0.0580,-0.0580);
  \fill[ruleink] (0.1740,-0.1740) rectangle ++(0.0580,-0.0580);
  \fill[ruleink] (0.2900,-0.1740) rectangle ++(0.0580,-0.0580);
  \fill[ruleink] (0.3480,-0.1740) rectangle ++(0.0580,-0.0580);
  \fill[ruleink] (0.4060,-0.1740) rectangle ++(0.0580,-0.0580);
  \fill[ruleink] (0.4640,-0.1740) rectangle ++(0.0580,-0.0580);
  \fill[ruleink] (-0.4640,-0.1740) rectangle ++(0.0580,-0.0580);
  \fill[ruleink] (-0.4060,-0.1740) rectangle ++(0.0580,-0.0580);
  \fill[ruleink] (-0.2900,-0.1740) rectangle ++(0.0580,-0.0580);
  \fill[ruleink] (-0.2320,-0.1740) rectangle ++(0.0580,-0.0580);
  \fill[ruleink] (-0.1740,-0.1740) rectangle ++(0.0580,-0.0580);
  \fill[ruleink] (-0.1160,-0.1740) rectangle ++(0.0580,-0.0580);
  \fill[ruleink] (0.0580,-0.2320) rectangle ++(0.0580,-0.0580);
  \fill[ruleink] (0.1160,-0.2320) rectangle ++(0.0580,-0.0580);
  \fill[ruleink] (0.2900,-0.2320) rectangle ++(0.0580,-0.0580);
  \fill[ruleink] (0.5220,-0.2320) rectangle ++(0.0580,-0.0580);
  \fill[ruleink] (-0.5220,-0.2320) rectangle ++(0.0580,-0.0580);
  \fill[ruleink] (-0.4640,-0.2320) rectangle ++(0.0580,-0.0580);
  \fill[ruleink] (-0.2900,-0.2320) rectangle ++(0.0580,-0.0580);
  \fill[ruleink] (-0.0580,-0.2320) rectangle ++(0.0580,-0.0580);
  \fill[ruleink] (0.0580,-0.2900) rectangle ++(0.0580,-0.0580);
  \fill[ruleink] (0.1740,-0.2900) rectangle ++(0.0580,-0.0580);
  \fill[ruleink] (0.2320,-0.2900) rectangle ++(0.0580,-0.0580);
  \fill[ruleink] (0.2900,-0.2900) rectangle ++(0.0580,-0.0580);
  \fill[ruleink] (0.3480,-0.2900) rectangle ++(0.0580,-0.0580);
  \fill[ruleink] (0.4640,-0.2900) rectangle ++(0.0580,-0.0580);
  \fill[ruleink] (0.5220,-0.2900) rectangle ++(0.0580,-0.0580);
  \fill[ruleink] (0.5800,-0.2900) rectangle ++(0.0580,-0.0580);
  \fill[ruleink] (-0.5800,-0.2900) rectangle ++(0.0580,-0.0580);
  \fill[ruleink] (-0.5220,-0.2900) rectangle ++(0.0580,-0.0580);
  \fill[ruleink] (-0.1160,-0.2900) rectangle ++(0.0580,-0.0580);
  \fill[ruleink] (-0.4060,-0.2900) rectangle ++(0.0580,-0.0580);
  \fill[ruleink] (-0.3480,-0.2900) rectangle ++(0.0580,-0.0580);
  \fill[ruleink] (-0.2900,-0.2900) rectangle ++(0.0580,-0.0580);
  \fill[ruleink] (-0.2320,-0.2900) rectangle ++(0.0580,-0.0580);
  \fill[ruleink] (-0.0580,-0.2900) rectangle ++(0.0580,-0.0580);
  \fill[ruleink] (0.0580,-0.3480) rectangle ++(0.0580,-0.0580);
  \fill[ruleink] (0.1740,-0.3480) rectangle ++(0.0580,-0.0580);
  \fill[ruleink] (0.4640,-0.3480) rectangle ++(0.0580,-0.0580);
  \fill[ruleink] (0.6380,-0.3480) rectangle ++(0.0580,-0.0580);
  \fill[ruleink] (-0.6380,-0.3480) rectangle ++(0.0580,-0.0580);
  \fill[ruleink] (-0.5800,-0.3480) rectangle ++(0.0580,-0.0580);
  \fill[ruleink] (-0.4060,-0.3480) rectangle ++(0.0580,-0.0580);
  \fill[ruleink] (-0.1160,-0.3480) rectangle ++(0.0580,-0.0580);
  \fill[ruleink] (0.0000,-0.4060) rectangle ++(0.0580,-0.0580);
  \fill[ruleink] (0.0580,-0.4060) rectangle ++(0.0580,-0.0580);
  \fill[ruleink] (0.1740,-0.4060) rectangle ++(0.0580,-0.0580);
  \fill[ruleink] (0.2320,-0.4060) rectangle ++(0.0580,-0.0580);
  \fill[ruleink] (0.4060,-0.4060) rectangle ++(0.0580,-0.0580);
  \fill[ruleink] (0.4640,-0.4060) rectangle ++(0.0580,-0.0580);
  \fill[ruleink] (0.5220,-0.4060) rectangle ++(0.0580,-0.0580);
  \fill[ruleink] (0.5800,-0.4060) rectangle ++(0.0580,-0.0580);
  \fill[ruleink] (0.6380,-0.4060) rectangle ++(0.0580,-0.0580);
  \fill[ruleink] (0.6960,-0.4060) rectangle ++(0.0580,-0.0580);
  \fill[ruleink] (-0.6960,-0.4060) rectangle ++(0.0580,-0.0580);
  \fill[ruleink] (-0.6380,-0.4060) rectangle ++(0.0580,-0.0580);
  \fill[ruleink] (-0.0580,-0.4060) rectangle ++(0.0580,-0.0580);
  \fill[ruleink] (-0.5220,-0.4060) rectangle ++(0.0580,-0.0580);
  \fill[ruleink] (-0.4640,-0.4060) rectangle ++(0.0580,-0.0580);
  \fill[ruleink] (-0.4060,-0.4060) rectangle ++(0.0580,-0.0580);
  \fill[ruleink] (-0.3480,-0.4060) rectangle ++(0.0580,-0.0580);
  \fill[ruleink] (-0.1740,-0.4060) rectangle ++(0.0580,-0.0580);
  \fill[ruleink] (-0.1160,-0.4060) rectangle ++(0.0580,-0.0580);
  \draw[rulegold!65, line width=0.3pt] (0.0000,0.0928) -- (0.0000,-0.5568);
  \draw[rulegold!65, line width=0.3pt] (0.0580,0.0928) -- (0.0580,-0.5568);
  \draw[rulecardline] (-0.8468,0.0928) rectangle (0.9048,-0.5568);
  \fill[rulegold] (-0.0261,-0.5568) -- (0.0841,-0.5568) -- (0.0290,-0.5133) -- cycle;
  \draw[white, line width=1.5pt] (0.0290,-0.4350) circle (0.1102);
  \draw[rulegold, line width=0.7pt] (0.0290,-0.4350) circle (0.1102);\end{tikzpicture} &
\begin{tikzpicture}[scale=2.8224]  \fill[rulecardfill] (-0.9048,0.0928) rectangle (0.9628,-0.5568);
  \fill[ruleband] (0.0000,0.0928) rectangle (0.0580,-0.5568);
  \fill[ruleink] (-0.3480,0.0000) rectangle ++(0.0580,-0.0580);
  \fill[ruleink] (0.3480,0.0000) rectangle ++(0.0580,-0.0580);
  \fill[ruleink] (0.2900,-0.0580) rectangle ++(0.0580,-0.0580);
  \fill[ruleink] (0.3480,-0.0580) rectangle ++(0.0580,-0.0580);
  \fill[ruleink] (0.4060,-0.0580) rectangle ++(0.0580,-0.0580);
  \fill[ruleink] (-0.4060,-0.0580) rectangle ++(0.0580,-0.0580);
  \fill[ruleink] (-0.3480,-0.0580) rectangle ++(0.0580,-0.0580);
  \fill[ruleink] (-0.2900,-0.0580) rectangle ++(0.0580,-0.0580);
  \fill[ruleink] (0.2320,-0.1160) rectangle ++(0.0580,-0.0580);
  \fill[ruleink] (0.2900,-0.1160) rectangle ++(0.0580,-0.0580);
  \fill[ruleink] (0.4640,-0.1160) rectangle ++(0.0580,-0.0580);
  \fill[ruleink] (-0.4640,-0.1160) rectangle ++(0.0580,-0.0580);
  \fill[ruleink] (-0.2320,-0.1160) rectangle ++(0.0580,-0.0580);
  \fill[ruleink] (-0.4060,-0.1160) rectangle ++(0.0580,-0.0580);
  \fill[ruleink] (0.1740,-0.1740) rectangle ++(0.0580,-0.0580);
  \fill[ruleink] (0.2320,-0.1740) rectangle ++(0.0580,-0.0580);
  \fill[ruleink] (0.3480,-0.1740) rectangle ++(0.0580,-0.0580);
  \fill[ruleink] (0.4060,-0.1740) rectangle ++(0.0580,-0.0580);
  \fill[ruleink] (0.4640,-0.1740) rectangle ++(0.0580,-0.0580);
  \fill[ruleink] (0.5220,-0.1740) rectangle ++(0.0580,-0.0580);
  \fill[ruleink] (-0.5220,-0.1740) rectangle ++(0.0580,-0.0580);
  \fill[ruleink] (-0.4640,-0.1740) rectangle ++(0.0580,-0.0580);
  \fill[ruleink] (-0.3480,-0.1740) rectangle ++(0.0580,-0.0580);
  \fill[ruleink] (-0.2900,-0.1740) rectangle ++(0.0580,-0.0580);
  \fill[ruleink] (-0.2320,-0.1740) rectangle ++(0.0580,-0.0580);
  \fill[ruleink] (-0.1740,-0.1740) rectangle ++(0.0580,-0.0580);
  \fill[ruleink] (0.1160,-0.2320) rectangle ++(0.0580,-0.0580);
  \fill[ruleink] (0.1740,-0.2320) rectangle ++(0.0580,-0.0580);
  \fill[ruleink] (0.3480,-0.2320) rectangle ++(0.0580,-0.0580);
  \fill[ruleink] (0.5800,-0.2320) rectangle ++(0.0580,-0.0580);
  \fill[ruleink] (-0.5800,-0.2320) rectangle ++(0.0580,-0.0580);
  \fill[ruleink] (-0.5220,-0.2320) rectangle ++(0.0580,-0.0580);
  \fill[ruleink] (-0.3480,-0.2320) rectangle ++(0.0580,-0.0580);
  \fill[ruleink] (-0.1160,-0.2320) rectangle ++(0.0580,-0.0580);
  \fill[ruleink] (0.0580,-0.2900) rectangle ++(0.0580,-0.0580);
  \fill[ruleink] (0.1160,-0.2900) rectangle ++(0.0580,-0.0580);
  \fill[ruleink] (0.2320,-0.2900) rectangle ++(0.0580,-0.0580);
  \fill[ruleink] (0.2900,-0.2900) rectangle ++(0.0580,-0.0580);
  \fill[ruleink] (0.3480,-0.2900) rectangle ++(0.0580,-0.0580);
  \fill[ruleink] (0.4060,-0.2900) rectangle ++(0.0580,-0.0580);
  \fill[ruleink] (0.5220,-0.2900) rectangle ++(0.0580,-0.0580);
  \fill[ruleink] (0.5800,-0.2900) rectangle ++(0.0580,-0.0580);
  \fill[ruleink] (0.6380,-0.2900) rectangle ++(0.0580,-0.0580);
  \fill[ruleink] (-0.6380,-0.2900) rectangle ++(0.0580,-0.0580);
  \fill[ruleink] (-0.5800,-0.2900) rectangle ++(0.0580,-0.0580);
  \fill[ruleink] (-0.1160,-0.2900) rectangle ++(0.0580,-0.0580);
  \fill[ruleink] (-0.4640,-0.2900) rectangle ++(0.0580,-0.0580);
  \fill[ruleink] (-0.4060,-0.2900) rectangle ++(0.0580,-0.0580);
  \fill[ruleink] (-0.3480,-0.2900) rectangle ++(0.0580,-0.0580);
  \fill[ruleink] (-0.2900,-0.2900) rectangle ++(0.0580,-0.0580);
  \fill[ruleink] (-0.1740,-0.2900) rectangle ++(0.0580,-0.0580);
  \fill[ruleink] (-0.0580,-0.2900) rectangle ++(0.0580,-0.0580);
  \fill[ruleink] (0.0580,-0.3480) rectangle ++(0.0580,-0.0580);
  \fill[ruleink] (0.2320,-0.3480) rectangle ++(0.0580,-0.0580);
  \fill[ruleink] (0.5220,-0.3480) rectangle ++(0.0580,-0.0580);
  \fill[ruleink] (0.6960,-0.3480) rectangle ++(0.0580,-0.0580);
  \fill[ruleink] (-0.6960,-0.3480) rectangle ++(0.0580,-0.0580);
  \fill[ruleink] (-0.6380,-0.3480) rectangle ++(0.0580,-0.0580);
  \fill[ruleink] (-0.4640,-0.3480) rectangle ++(0.0580,-0.0580);
  \fill[ruleink] (-0.1740,-0.3480) rectangle ++(0.0580,-0.0580);
  \fill[ruleink] (0.0000,-0.4060) rectangle ++(0.0580,-0.0580);
  \fill[ruleink] (0.0580,-0.4060) rectangle ++(0.0580,-0.0580);
  \fill[ruleink] (0.1160,-0.4060) rectangle ++(0.0580,-0.0580);
  \fill[ruleink] (0.1740,-0.4060) rectangle ++(0.0580,-0.0580);
  \fill[ruleink] (0.2320,-0.4060) rectangle ++(0.0580,-0.0580);
  \fill[ruleink] (0.2900,-0.4060) rectangle ++(0.0580,-0.0580);
  \fill[ruleink] (0.4640,-0.4060) rectangle ++(0.0580,-0.0580);
  \fill[ruleink] (0.5220,-0.4060) rectangle ++(0.0580,-0.0580);
  \fill[ruleink] (0.5800,-0.4060) rectangle ++(0.0580,-0.0580);
  \fill[ruleink] (0.6380,-0.4060) rectangle ++(0.0580,-0.0580);
  \fill[ruleink] (0.6960,-0.4060) rectangle ++(0.0580,-0.0580);
  \fill[ruleink] (0.7540,-0.4060) rectangle ++(0.0580,-0.0580);
  \fill[ruleink] (-0.7540,-0.4060) rectangle ++(0.0580,-0.0580);
  \fill[ruleink] (-0.6960,-0.4060) rectangle ++(0.0580,-0.0580);
  \fill[ruleink] (-0.5800,-0.4060) rectangle ++(0.0580,-0.0580);
  \fill[ruleink] (-0.5220,-0.4060) rectangle ++(0.0580,-0.0580);
  \fill[ruleink] (-0.4640,-0.4060) rectangle ++(0.0580,-0.0580);
  \fill[ruleink] (-0.4060,-0.4060) rectangle ++(0.0580,-0.0580);
  \fill[ruleink] (-0.2320,-0.4060) rectangle ++(0.0580,-0.0580);
  \fill[ruleink] (-0.1740,-0.4060) rectangle ++(0.0580,-0.0580);
  \fill[ruleink] (-0.1160,-0.4060) rectangle ++(0.0580,-0.0580);
  \draw[rulegold!65, line width=0.3pt] (0.0000,0.0928) -- (0.0000,-0.5568);
  \draw[rulegold!65, line width=0.3pt] (0.0580,0.0928) -- (0.0580,-0.5568);
  \draw[rulecardline] (-0.9048,0.0928) rectangle (0.9628,-0.5568);
  \fill[rulegold] (-0.0261,-0.5568) -- (0.0841,-0.5568) -- (0.0290,-0.5133) -- cycle;
  \draw[white, line width=1.5pt] (0.0290,-0.4350) circle (0.1102);
  \draw[rulegold, line width=0.7pt] (0.0290,-0.4350) circle (0.1102);\end{tikzpicture}
\end{tabular}
\\[3mm]
\caption{The zero-center horizon is flat across each pair $w=2k-1,2k$, the
parity effect of Theorem~\ref{thm:horizon}: raising the radius from odd to even
buys no extra time.  One extremizer per radius $w=1,\ldots,6$, truncated from
the zero-trace completion at left depth $w$.  Time runs downward, and the
circled cell is the first nonzero center, occurring at time
$2\lceil w/2\rceil+1$.}
\label{fig:horizon}
\end{figure}
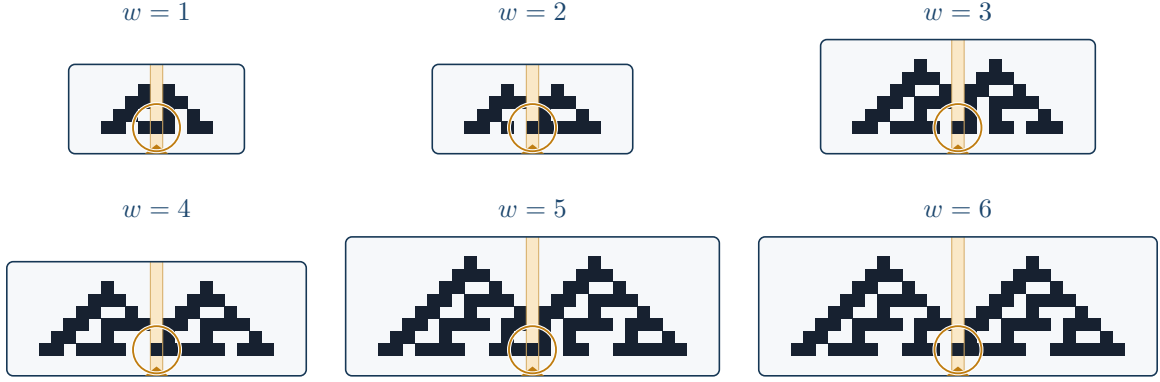

\section{A sharp finite horizon}

Corollary~\ref{cor:finite} says a finite row's center cannot stay zero forever.
How long it can stay zero, as a function of the support radius, is exactly one
step past the light cone of the first odd depth beyond that radius.

For $w\geq1$, let
$\mathcal S_w=\{x\neq0:\supp(x)\subseteq[-w,w],\ x_0=0\}$.  For
$x\in\mathcal S_w$, the zero-center horizon is defined as follows.

\begin{equation}\label{eq:horizon}
 H(x)=\max\{h:F^t(x)_0=0\text{ for }0\leq t\leq h\}
\end{equation}

Thus $H(x)$ is the last matching time index; the corresponding prefix contains
$H(x)+1$ symbols, including the symbol at time zero.

\begin{theorem}[Sharp finite horizon]\label{thm:horizon}
Let $w\geq1$.  Then both of the following hold.

\begin{equation}\label{eq:sharp}
 \max_{x\in\mathcal S_w}H(x)=2\left\lceil\frac w2\right\rceil,
 \qquad
 \#\!\left\{x\in\mathcal S_w:H(x)=2\left\lceil\frac w2\right\rceil\right\}
 =2^w-1
\end{equation}
\end{theorem}
\begin{proof}
Let $q$ be the least odd integer greater than $w$, so
$q-1=2\lceil w/2\rceil$.  If the positive right half is nonzero, a trace zero
through time $q$, compared with its zero-trace completion, would force by
finite-prefix uniqueness and \eqref{eq:odd}
$x_{-q}=\bigvee_{j=1}^qR_j=1$, contradicting $q>w$.  If that half is zero,
comparison with the zero row makes every nonzero left cell visible by time
$w$.  Thus $H(x)\leq q-1$.

Conversely, choose any of the $2^w-1$ nonzero words
$(R_1,\ldots,R_w)$, install \eqref{eq:odd}--\eqref{eq:even} through left
depth $w$, and zero all other cells.  This row agrees with its infinite
zero-trace completion through the light cone of time $q-1$; when $w$ is odd,
the extra depth $w+1$ is even and forced to zero.  The missing odd cell at
depth $q$ creates a discrepancy at time $q$, and finite-prefix uniqueness
excludes further extremizers.
\end{proof}

The one-trace horizon is the same computation against \eqref{eq:checker}, and
the parities exchange roles: ones sit at even depths, so the binding depth is
the least \emph{even} one past the radius, and $R_0=1$ already makes the row
nonzero, which is why no extremizer is lost here.

\begin{theorem}[Sharp finite horizon, one]\label{thm:onehorizon}
Let $w\geq1$, let
$\mathcal T_w=\{x:\supp(x)\subseteq[-w,w],\ x_0=1\}$, and for
$x\in\mathcal T_w$ let
$H_1(x)=\max\{h:F^t(x)_0=1\text{ for }0\leq t\leq h\}$.  Let $q$ be the
least even integer greater than $w$.  Then both of the following hold.

\begin{equation}\label{eq:sharpone}
 \max_{x\in\mathcal T_w}H_1(x)=q-1,
 \qquad
 \#\!\left\{x\in\mathcal T_w:H_1(x)=q-1\right\}=2^w
\end{equation}
\end{theorem}
\begin{proof}
A trace one through time $q$, compared with the all-one completion of
Theorem~\ref{thm:onefiber}, would force by finite-prefix uniqueness and
\eqref{eq:checker} $x_{-q}=1$, contradicting $q>w$; thus $H_1(x)\leq q-1$.
Conversely, choose any of the $2^w$ words $(R_1,\ldots,R_w)$, install
\eqref{eq:checker} through left depth $w$, and zero all other cells.  This row
agrees with its all-one completion through the light cone of time $q-1$; when
$w$ is even, the extra depth $w+1$ is odd and forced to zero.  The missing even
cell at depth $q$ creates a discrepancy at time $q$, and finite-prefix
uniqueness excludes further extremizers.
\end{proof}

Combining the two horizons, the sharp law over both constants is flat in the
parity that splits each one separately.

\begin{corollary}\label{cor:bothhorizon}
Let $w\geq1$ and let $x\neq0$ have $\supp(x)\subseteq[-w,w]$.  The longest
constant prefix of the central trace of $x$ has length at most $w+2$, and
exactly $2^w$ configurations attain length $w+2$ when $w$ is even, while
exactly $2^w-1$ attain it when $w$ is odd.
\end{corollary}
\begin{proof}
For $x_0=0$ Theorem~\ref{thm:horizon} gives $2\lceil w/2\rceil$, which is $w$
for even $w$ and $w+1$ for odd $w$; for $x_0=1$
Theorem~\ref{thm:onehorizon} gives $q-1$, which is $w+1$ for even $w$ and $w$
for odd $w$.  Hence the maximum horizon index is $w+1$, so the maximum prefix
length is $w+2$.  It is attained in the one case when $w$ is even and in the
zero case when $w$ is odd, with the corresponding counts $2^w$ and $2^w-1$.
\end{proof}

\section{Conclusion}

Theorems~\ref{thm:fiber} and~\ref{thm:onefiber} classify the two constant fibers
of the trace map, and Corollaries~\ref{cor:finite} and~\ref{cor:constant} show
that neither meets the finite configurations away from the origin.  The next
unresolved case is eventual period two; we make no claim about it or any higher
period.

One structural fact bounds what this method reaches.  Write $H(p,w)$ for the
sharp maximum length of an eventually $p$-periodic central prefix over rows of
support radius $w$, so that Theorems~\ref{thm:horizon}
and~\ref{thm:onehorizon} give $H(1,w)=w+2$.  At $p=2$ no bounded law can exist:
left permutivity supplies, for any prescribed trace of length $N$ and any
compatible right prefix, a unique left half and hence a finite row of support
radius $\approx N$ realizing that trace through time $N$, so $H(2,w)\ge w$ for
every $w$.  The constant case is therefore not the first of a uniform family.
The $p=1$ horizon is finite and sharp in $w$; the $p=2$ horizon grows at least
linearly, and a structural account of period-two exclusion remains open.

\appendix
\section{Reproducibility and formalization}\label{app:repro}

The theorems above are proved analytically.  The computations recorded here
corroborate them and generate the figures; no proof depends on them.

An independent numerical audit reproduces the sharp horizon and the extremizer
counts of Theorems~\ref{thm:horizon} and~\ref{thm:onehorizon}.  For support
radius $w$, write $H_0$ and $H_1$ for the maximum horizon index
\eqref{eq:horizon} when the initial center is zero and one respectively, each
with the number of configurations attaining it.

\begin{center}
\small
\begin{tabular}{@{}rrrrrrr@{}}
\toprule
$w$ & $H_0$ & count & $H_1$ & count & $\max H$ & count\\
\midrule
1&2&1&1&2&2&1\\
2&2&3&3&4&3&4\\
3&4&7&3&8&4&7\\
4&4&15&5&16&5&16\\
5&6&31&5&32&6&31\\
6&6&63&7&64&7&64\\
7&8&127&7&128&8&127\\
\bottomrule
\end{tabular}
\end{center}

Prefix lengths are these indices plus one, so $\max H+1=w+2$ at every radius.
The two horizons alternate which parity leads: $H_0$ is larger for odd $w$,
$H_1$ for even $w$.  The final column confirms
Corollary~\ref{cor:bothhorizon}, $2^w$ extremizers for even $w$ and $2^w-1$ for
odd $w$.

Every cell in Figures~\ref{fig:panels} and~\ref{fig:horizon} is generated from
\eqref{eq:rule30}.  Figure~\ref{fig:panels} exhibits an all-zero center for the
displayed $\mathcal C_3$ row, an all-zero Rule-90 center for the row with ones
at $\{-1,1\}$, and escape of that same row under Rule 30.  The six panels of
Figure~\ref{fig:horizon}, for $w=1,\ldots,6$, have first nonzero center times
$3,3,5,5,7,7$.

The accompanying \texttt{Rule30ZeroTail.lean} is a partial formalization.  It
formalizes the local equation and left permutivity, propagation of a rightmost
discrepancy and triangular unit sensitivity, the prefix-OR closed form assuming
a first positive right-hand one, the local Boolean identities for the
$\mathcal C_m$ invariant, and the non-left-bounded conclusion conditional on
the classification disjunction.  It does \emph{not} derive that disjunction from
an all-zero trace, and it does not formalize the all-one fiber, the finite-support
corollaries, or either sharp horizon and extremizer count.  The formalization is
therefore supporting verification, not a machine-checked proof of the main
theorems.

\paragraph{Code availability.}
The audit script, the figure generator, the executable tests and the partial
Lean formalization accompany this submission as ancillary files.

\paragraph{Use of AI tools.}
The model \texttt{anthropic/claude-opus-5} assisted with drafting, with the audit
and figure scripts, and with locating references.  The author derived and
checked every definition, theorem, proof, and numerical value independently of
these tools, verified every reference against its primary source, and takes full
responsibility for all contents of this paper irrespective of how they were
generated.

\section*{Acknowledgments}
I thank Jarkko Kari for reading an earlier version of this note, and for
suggesting the exclusion of nonconstant periodic traces as the natural next
step.

\end{document}